\documentclass{amsart}

\usepackage{amssymb,amsfonts,amsxtra}
\usepackage{paralist}
\usepackage{dsfont,mathrsfs}
\renewcommand{\mathcal}{\mathscr}
\usepackage[all]{xypic}
\xyoption{dvips}

\theoremstyle{definition}
\newtheorem{ntn}{Notation}

\theoremstyle{plain}
\newtheorem{lem}[ntn]{Lemma}

\newtheorem{thm}[ntn]{Theorem}
\newtheorem{cor}[ntn]{Corollary}

\theoremstyle{remark}

\newtheorem{rmk}[ntn]{Remark}

\newcommand{\calO}{\mathcal{O}}
\newcommand{\CC}{\mathds{C}}
\newcommand{\gl}{\mathfrak{gl}}

\newcommand{\ideal}[1]{{\left\langle#1\right\rangle}}
\newcommand{\into}{\hookrightarrow}

\newcommand{\mm}{\mathfrak{m}}

\newcommand{\onto}{\twoheadrightarrow}
\newcommand{\p}{\partial}
\renewcommand{\sl}{\mathfrak{sl}}
\renewcommand{\ss}{\mathfrak{s}}

\DeclareMathOperator{\Der}{Der}
\DeclareMathOperator{\edim}{embdim}

\DeclareMathOperator{\ord}{ord}

\begin{document}

\title[The Lie algebra of derivations of a fat point]
{A solvability criterion for the Lie algebra of derivations of a fat point}

\author{Mathias Schulze}
\address{
M.~Schulze\\
Oklahoma State University\\
401 Mathematical Sciences\\
Stillwater, OK 74078\\
USA}
\email{mschulze@math.okstate.edu}
\thanks{The author was supported by the College of Arts \& Sciences at Oklahoma State University through a FY 2010 ASR+1 award.}

\date{\today}



\begin{abstract}
We consider the Lie algebra of derivations of a zero dimensional local complex algebra.
We describe an inequality involving the embedding dimension, the order, and the first deviation that forces this Lie algebra to be solvable.
Our result was motivated by and generalizes the solvability of the Yau algebra of an isolated hypersurface singularity.
\end{abstract}

\maketitle

\section{Introduction}

Let $f\in\mm=\ideal{x_1,\dots,x_n}\subseteq\CC\{x_1,\dots,x_n\}=\calO$ define an isolated hypersurface singularity $X=\{f=0\}\subseteq(\CC^n,0)$.
The \emph{Tyurina algebra} of $X$ is the finite $\CC$-algebra $A(X)=\calO/\ideal{f,J(f)}$ where $J(f)=\ideal{\p f/\p x_1,\dots,\p f/\p x_n}$.
By a result of Mather and Yau \cite{MY82}, $A(X)$ determines the analytic isomorphism class of $X$.
The \emph{Yau algebra} $L(X)=\Der_\CC A(X)$ is the Lie algebra of derivations of $A(X)$.
Its structure and the interplay of its properties with those of $X$ do not seem to be understood in general.
For instance, there is the \emph{recognition problem} asking what Lie algebras can arise as $L(X)$, the \emph{recognition problem} asking what information on $X$ can be restored from $L(X)$, and in particular the \emph{classification problem} asking up to what extend $L(X)$ determines the isomorphism class of $A(X)$ and hence of $X$. 
In the case of simple singularities, the classification problem has been studied by Elashvili and Khimshiashvili \cite{EK06}.

An important result on the recognition problem was formulated by Yau~\cite[Thm.~2]{Yau91}: $L(X)$ is a solvable complex Lie algebra. 
The purpose of this note is to prove the following generalization of this result in which we replace $J$ by a general zero dimensional ideal.
Our approach is inspired by the result of M\"uller \cite[Hilfssatz~2]{Mue86} that any ideal of an analytic algebra invariant under a reductive algebraic group is minimally generated by an invariant vector space.

\begin{thm}\label{1}
Let $S$ be a zero-dimensional local $\CC$-algebra of embedding dimension $\edim(S)$ and order $\ord(S)$, and denote by $\varepsilon_1(S)$ its first deviation.
Then the Lie algebra $\Der_\CC S$ is solvable if $\varepsilon_1(S)+1<\edim(S)+\ord(S)$.
\end{thm}

Recall that, by definition, $\varepsilon_1(S)=\dim_\CC H_1(S)$ where $H_\bullet(S)$ is the Koszul algebra of $S$.
More explicitly, Theorem~\ref{1} applies to $S=R/I$ where $R=\CC[\![x_1,\dots,x_n]\!]$ and $I\subseteq R$ is a zero dimensional ideal with $I\subseteq\mm^m$ where $\mm=\ideal{x_1,\dots,x_n}$ and $m\ge2$ is chosen maximal.
Then $n=\edim(S)$, $m=\ord(S)$, and $\varepsilon_1(S)=\dim_\CC(I/\mm I)$ is the minimal number of generators of $I$ \cite[Thm.~2.3.2.(b)]{BH93}.

By \cite[Thm.~2.3.3.(b)]{BH93}, $S$ being a complete intersection is equivalent to $\varepsilon_1(S)=\edim(S)-\dim(S)$.
In this case, the inequality in Theorem~\ref{1} reduces to $1<\ord(S)$ which holds trivially unless $S=\CC$.

\begin{cor}
If $S$ is a zero-dimensional complex complete intersection then $\Der_\CC S$ is a solvable Lie algebra.
\end{cor}

This result applies in particular to the \emph{Milnor algebra} $S=\calO/J(f)$ of a a function $f\colon(\CC^n,0)\to(\CC,0)$ with isolated critical point.

\begin{cor}
If $S$ is a Milnor algebra then $\Der_\CC S$ is a solvable Lie algebra.
\end{cor}

Using a theorem of Kempf~\cite[Thm.~13]{Kem93} in the order $3$ case, also the solvability of Yau algebras becomes a corollary of Theorem~\ref{1}.

\begin{cor}[Yau's solvability theorem]\label{2}
The Yau algebra of any isolated hypersurface singularity is a solvable Lie algebra.
\end{cor}

\begin{rmk}
We do not understand the ``induction step'' in Yau's proof of Corollary~\ref{2} \cite[Thm.~2]{Yau91}: 
Yau considers the Taylor expansion $f=\sum_{i=k+1}^\infty f_i$ of $f$ and assumes $A(f)$ to be $\ss$-invariant for some $\sl_2(\CC)\cong\ss\subseteq\gl_n(\CC)$.
By hypothesis, $J(f)$ contains $\mm^m$ for some $m\ge k$ and, by finite determinacy, one can assume that $f$ is a polynomial of degree $d\le m$.
Then also the larger ideal $J(f_{k+1})+\cdots+J(f_d)$ of homogeneous parts of $J(f)$ contains $\mm^m$ and hence the quotient $(\mm^\ell+J(f_{k+1})+\cdots+J(f_\ell))/(\mm^{\ell+1}+J(f_{k+1})+\cdots+J(f_\ell))$ is zero for any $\ell\ge m$.
However, Yau identifies a subspace $J_\ell$ of this zero space with the vector space $\ideal{\p f_{\ell+1}/\p x_1,\dots,\p f_{\ell+1}/\p x_n}_\CC$ in order to conclude that the latter is $\ss$-invariant.
We do not see any reason why this space should be $\ss$-invariant in general for $\ell>k$.

However, for $\ell=k$, this invariance holds true and Yau's argument \cite[Thm.~1]{Yau91} using Kempf's result~\cite[Thm.~13]{Kem93} proves Corollary~\ref{2} in the homogeneous case (cf.~Remark~\ref{21}).
We shall use the same idea to prove the order $3$ case in general (cf.~Remark~\ref{12}).
In the higher order cases, we shall apply our main result Theorem~\ref{1}.
\end{rmk}

\section{Proofs}

For an ideal $J$ in an analytic algebra $S$, denote by $\Der_J S\subseteq\Der_\CC S$ the Lie subalgebra of all $\delta\in\Der_\CC S$ for which $\delta(J)\subseteq J$.
We shall first reformulate the claim of Theorem~\ref{1} in terms of the Lie algebra $\Der_{\Bar I}\Bar R$ where $\Bar R=R/\mm^{\ell+1}$, for sufficiently large $\ell\ge2$, and $\Bar I$ is the image of $I$ in $\Bar R$. 
To this end, we shall use the following result.

\begin{lem}\label{15}
For $J$ be an ideal in $R=\CC[\![x_1,\dots,x_n]\!]$.
Then there is a natural isomorphism of Lie algebras
\[
(\Der_JR)/(J\cdot\Der_\CC R)\cong\Der_\CC(R/J).
\]
\end{lem}

\begin{proof}
By definition, there is a map $\varphi\colon\Der_JR\to\Der_\CC(R/J)$ whose kernel contains $J\cdot\Der_\CC R$.
Note that $\Der_\CC R$ is a free $R$-module with basis $\p/\p x_1,\dots,\p/\p x_n$ and that the coefficient of $\p/\p x_i$ in $\delta\in\Der_\CC R$ is $\delta(x_i)$.
So if $\delta\in\ker\varphi$, then $\delta(x_i)\in J$ and hence $\delta\in J\cdot\Der_\CC R$.
This proves injectivity.
By a result of Scheja and Wiebe \cite[(2.1)]{SW73}, any $\Bar\delta\in\Der_\CC(R/J)$ lifts to a $\delta\in\Der_\CC R$ which is then necessarily in $\Der_JR$.
This proves surjectivity and the claim follows.
\end{proof}

Consider now the situation of Theorem~\ref{1}.
We shall use the notation introduced in the paragraph following Theorem~\ref{1}.
Since $I$ is $\mm$-primary by hypothesis,
\begin{equation}\label{16}
\mm^\ell\subseteq I,\text{ for some }\ell\ge2,
\end{equation}
and we set
\begin{equation}\label{18}
\Bar I:=I/\mm^{\ell+1}\subseteq\Bar\mm:=\mm/\mm^{\ell+1}\subseteq\Bar R:=R/\mm^{\ell+1}.
\end{equation}
Note that
\begin{equation}\label{11}
\Bar I/(\Bar\mm\cdot\Bar I)\cong I/(\mm\cdot I)
\end{equation}
and hence $\Bar I$ has the same minimal number of generators as $I$. 
As $\mm$ is an associated prime of $I$, it follows from a result of Scheja and Wiebe \cite[(2.5)]{SW73} that
\begin{equation}\label{3}
\Der_IR\subseteq\Der_\mm R.
\end{equation} 
Using the Leibniz rule and \cite[(2.5)]{SW73} again, one shows that
\begin{equation}\label{9}
\Der_\mm R=\Der_{\mm^i}R,\text{ for all }i\ge2.
\end{equation}
Using \cite[Ch.~I, \S7, Lem.]{Jac79}, the following result reduces the claim of Theorem~\ref{1} to prove solvability of $\Der_{\Bar I}\Bar R$.

\begin{lem}\label{17}
$\Der_\CC(S)$ a subquotient of $\Der_{\Bar I}\Bar R$.
\end{lem}

\begin{proof}
Using~\eqref{9}, Lemma~\ref{15} yields a natural isomorphism
\begin{equation}\label{14}
(\Der_\mm R)/(\mm^{\ell+1}\cdot\Der_\CC R)\cong\Der_\CC\Bar R.
\end{equation}
By the choice of $\ell$ and \eqref{3},
\[
\mm^{\ell+1}\cdot\Der_\CC R\subseteq\mm^\ell\cdot\Der_\CC R\subseteq\Der_IR\subseteq\Der_\mm R
\]
and hence \eqref{14} induces an injection 
\begin{equation}\label{5}
(\Der_IR)/(\mm^{\ell+1}\cdot\Der_\CC R)\into\Der_{\Bar I}\Bar R.
\end{equation}
Lemma~\ref{15} also yields an isomorphism
\begin{equation}\label{6}
(\Der_IR)/(I\cdot\Der_\CC R)\cong\Der_\CC(R/I).
\end{equation}
Combining \eqref{5} and \eqref{6} makes $\Der_\CC(R/I)$ a subquotient of $\Der_{\Bar I}\Bar R$ as claimed.
\end{proof}

In order to apply the reduction given by Lemma~\ref{17}, we shall need \eqref{3} and \eqref{9} also for $R$ replaced by $\Bar R$.
By \cite[(2.5)]{SW73}, any $\delta\in\Der_{\Bar I}\Bar R$ lifts to a $\delta'\in\Der_{\CC} R$. 
By \eqref{16}, \eqref{18}, and \eqref{3}, $\delta'\in\Der_I R=\Der_\mm R$ and hence $\delta\in\Der_{\Bar\mm}\Bar R$.
Thus, 
\begin{equation}\label{20}
\Der_{\Bar I}\Bar R\subseteq\Der_{\Bar\mm}\Bar R
\end{equation}
analogous to \eqref{3}, and, using the the Leibnitz rule, one deduces
\begin{equation}\label{4}
\Der_{\Bar I}\Bar R\subseteq\Der_{\Bar\mm\cdot\Bar I}\Bar R.
\end{equation}
The analogue of \eqref{9} is proved similarly and reads
\begin{equation}\label{19}
\Der_\CC\Bar R=\Der_{\Bar\mm^i}\Bar R,\text{ for all }i\ge1.
\end{equation}

After these preparations we are ready to finish the 
\begin{proof}[Proof of Theorem~\ref{1}]
Assume that $\Der_\CC S$ is not solvable which implies, by Lemma~\ref{17}, that $\Der_{\Bar I}\Bar R$ is not solvable.
By Levi's theorem \cite[Ch.~III, \S9, Levi's thm.]{Jac79} and the structure of semisimple Lie algebras \cite[Ch.~IV, \S3]{Jac79}, there is a Lie subalgebra
\[
\sl_2(\CC)\cong\ideal{H,X,Y}=\ss\subseteq\Der_{\Bar I}\Bar R
\]
where $H,X,Y$ are standard generators subject to the relations
\[
[H,X]=2X,\quad[H,Y]=-2Y,\quad[X,Y]=H.
\]
By \eqref{20}, \eqref{19} for $i=2$, and complete reducibility \cite[Ch.~III, \S7, Thm.~8]{Jac79}, we can assume that the $\CC$-vector space $\Bar V=\ideal{\Bar x_1,\dots,\Bar x_n}_\CC$ spanned by the $\Bar x_i=x_i+\mm^{\ell+1}\in\Bar R$ is an $\ss$-invariant vector space.
This means that the representation of the Lie algebra $\ss$ on $\Bar R$ is linear with respect to the coordinates $\Bar x_1,\dots,\Bar x_n$.
By the same argument using \eqref{4} instead of \eqref{9} for $i=2$, we can find an $\ss$-invariant vector space $\Bar F\subseteq\Bar  I$ that maps isomorphically onto $\Bar I/(\Bar\mm\cdot\Bar I)$.
In other words, a basis of $\Bar F$ is a minimal set of generators of $\Bar I$.
Then, by \eqref{11} and according to the interpretation following Theorem~\ref{1}, we have to show that
\begin{equation}\label{10}
\dim_\CC\Bar F\ge m+n-1.
\end{equation}

Now, again by complete reducibility and by the classification of irreducible modules \cite[Ch.~VII, \S3]{Jac79}, $\Bar V$ decomposes into irreducible $\ss$-modules each of which is generated by a highest weight vector, which means an $H$-homogeneous vector in $\ker X$.
We may assume that $\Bar x_1,\dots,\Bar x_k$ are these highest weight vectors, that is, $X(x_i)=0$ for $i=1,\dots,k$, and that their $H$-weights form a nonincreasing sequence, that is, $H(x_1)/x_1\ge\cdots\ge H(x_k)/x_k$.
Moreover, we can choose the remaining variables $\Bar x_{k+1},\dots,\Bar x_n$ to be $H$-homogeneous and compatible with the decomposition of $\Bar V$ into irreducible $\ss$-modules, that is, $X$ and $Y$ map variables to multiples of variables.
Since $I$ is zero dimensional,
\[
(I+\ideal{x_{i+1},\dots,x_n})/\ideal{x_{i+1},\dots,x_n}\subseteq R/\ideal{x_{i+1},\dots,x_n}\cong\CC[\![x_1,\dots,x_i]\!]
\]
is a zero dimensional ideal in an $i$-dimensional power series ring.
So it must have at least $i$ generators.
By the analogue of \eqref{11} with $R$ replaced by $R/\ideal{x_{i+1},\dots,x_n}$, the same holds for $(\Bar I+\ideal{\Bar x_{i+1},\dots,\Bar x_n})/\ideal{\Bar x_{i+1},\dots,\Bar x_n}$ which is generated by the image of $\Bar F$ in $\Bar R/\ideal{\Bar x_{i+1},\dots,\Bar x_n}$.
Therefore we can associate to each $i=1,\dots,k$ an $H$-homogeneous element $\Bar f_i\in\Bar F$ such that $\Bar f_1,\dots,\Bar f_i$ are linearly independent modulo $\ideal{\Bar x_{i+1},\dots,\Bar x_n}$.
These elements generate an $H$-homogeneous $k$-dimensional vector space 
\[
\Bar F'=\ideal{\Bar f_1,\dots,\Bar f_k}_\CC.
\]
As $\Bar x_1,\dots,\Bar x_k$ are highest weight vectors in $\Bar V$ and hence in $\Bar R$, the subring $\CC[\Bar x_1,\dots,\Bar x_k]\subseteq\Bar R$ consists of highest weight vectors by the Leibnitz rule.
Thus,
\begin{equation}\label{8}
\Bar g_i(\Bar x_1,\dots,\Bar x_k)=\Bar f_i(\Bar x_1,\dots,\Bar x_k,0,\dots,0),\text{ for }i=1,\dots,k,
\end{equation}
are linearly independent highest weight vectors in $\Bar R$ and
\[
\Bar G=\ideal{\Bar g_1,\dots,\Bar g_k}_\CC
\]
is an $H$-homogeneous vector space.
Any highest weight vector outside of $\CC[\Bar x_1,\dots,\Bar x_k]$ can be chosen to lie in the ideal $\ideal{\Bar x_{k+1},\dots,\Bar x_n}\subseteq\Bar R$.
As the latter is stable by the Borel algebra $\ideal{H,Y}_\CC$, it contains also the $\ss$-module generated by such a vector.
This shows that there is a projection of $\ss$-modules
\[
\Bar F\supseteq\ss\cdot\Bar F'\onto\ss\cdot\Bar G.
\]
So in order to find a lower bound for $\dim_\CC\Bar F$ such as \eqref{10}, we may as well assume that $\Bar f_i=\Bar g_i$ for $i=1,\dots,k$.

Recall the hypothesis $I\subseteq\mm^m$, $m\ge2$, which implies $\Bar I\subseteq\Bar\mm^m$, and denote by $m_i$ the $H$-weight of $\Bar x_i$.
Then $\Bar g_i$ has $H$-weight at least $m\cdot m_i$ and hence $\ss\cdot\Bar g_i$ has dimension at least $m\cdot m_i+1$.
Note that the dimension of $\Bar V$ is $n=m_1+1+\cdots+m_k+1$.
It follows from the preceding arguments that
\begin{equation}\label{7}
\dim_\CC\Bar F\ge(m-1)\cdot(m_1+\cdots+m_k)+n.
\end{equation}
As $\ss\subseteq\gl_n(\CC)$, at least one of the irreducible $\ss$-modules in $\Bar V$ must be nontrivial which means that $m_i\ge1$ for some $i\in\{1,\dots,k\}$.
It follows that $\dim_\CC\Bar F\ge m+n-1$ as claimed.
This finishes the proof of Theorem~\ref{1}.
\end{proof}

\begin{rmk}\label{21}
In the setting of reductive groups of automorphisms of analytic algebras, an invariant set of minimal generators such as $\Bar F$ in the proof of Theorem~\ref{1} has been constructed by M\"uller \cite[Hilfssatz~2]{Mue86}. 
\end{rmk}

Finally we give the 
\begin{proof}[Proof of Corollary~\ref{2}]
By a remark of Yau~\cite[\S2 Rem.]{Yau91}, one can reduce to the case where $f\in\mm^3$.
Theorem~\ref{1} applied to $I=\ideal{f,J(f)}$ yields the claim 
\begin{itemize}
\item if $\ord(X)\ge4$, 
\item if $X$ is quasihomogeneous and hence $I=J(f)$, or
\item if $\Bar V$ contains a $3$-dimensional $\ss$-module using \eqref{7}.
\end{itemize}
So we may assume that $f\in\mm^3\setminus\mm^4$ and that $\Bar V$ contains exactly one $2$-dimensional and $n-2$ many $1$-dimensional irreducible $\ss$-modules, say $\ideal{\Bar x_1,\Bar x_n}_\CC$ and $\ideal{\Bar x_2}_\CC,\dots,\ideal{\Bar x_{n-1}}_\CC$.
In particular, the $H$-weights of $\Bar x_1,\Bar x_2,\dots,\Bar x_{n-1},\Bar x_n$ equal $m_1=1,m_2=0,\dots,m_k=0,-1$.
Denote by $\Bar f^3\in\Bar V^3$ the homogeneous part of degree $3$ of $\Bar f=f+\mm^{\ell+1}\in\Bar R$ with respect to $\Bar x_1,\dots,\Bar x_n$.
Then $J(\Bar f^3)\subseteq\Bar V^2$ is an $\ss$-module.
By a result of Kempf~\cite[Thm.~13]{Kem93}, there is an $\ss$-invariant polynomial $\Bar g\in\Bar V^3$ such that $J(\Bar g)=J(\Bar f^3)$.
In particular, $\Bar g$ must be $H$-homogeneous of weight $0$ and hence in the span of $x_1\cdot\ideal{x_2,\dots,x_{n-1}}_\CC\cdot x_n$ and $\ideal{x_2,\dots,x_{n-1}}_\CC^3$.
But then $X(\Bar g)=0=Y(\Bar g)$ forces $\Bar g$ and hence $J(\Bar f^3)$ to be independent of $\Bar x_1$ and $\Bar x_n$.
This shows that $\Bar g_1$ in \eqref{8} involves at least a third power of $\Bar x_1$.
Thus, the inequality \eqref{7} can be improved by setting $m=3$ which suffices to conclude Corollary~\ref{2}.
\end{proof}

\begin{rmk}\label{12}
We do not know how to avoid the rather deep result of Kempf in the proof of the order $3$ case in Corollary~\ref{2}.
\end{rmk}

\bibliographystyle{amsalpha}
\bibliography{dfp}
\end{document}